\documentclass[a4paper,11pt]{amsart}
\usepackage[
  top=3cm,
  bottom=3cm,
  left=3cm,
  right=3cm,
  headheight=14pt,
  headsep=1cm,
  footskip=1cm
]{geometry}
\usepackage{amsmath, epsfig, verbatim}
\usepackage{amsfonts}
\usepackage{amsthm}
\usepackage{mathtools}
\usepackage{amsmath}
\usepackage{amssymb} 
\usepackage{graphics,graphicx}
\usepackage{epstopdf}
\usepackage{listings}
\usepackage{xcolor}
\usepackage{tikz}
\usepackage{subcaption}
\usepackage{pgfplots}
\usepackage{caption}
\usepackage{booktabs}
\usepackage{tabularx}
\usepackage{hyperref}   
\usepackage{doi}
\usepackage{enumitem}
\usepackage[numbers,sort&compress]{natbib}
\usetikzlibrary{calc}
\allowdisplaybreaks
\definecolor{codegray}{rgb}{0.95,0.95,0.95}
\definecolor{pykeyword}{rgb}{0.13,0.13,1}
\definecolor{pystring}{rgb}{0.58,0,0.82}

\lstdefinestyle{pythonstyle}{
    backgroundcolor=\color{codegray},
    language=Python,
    basicstyle=\ttfamily\small,
    keywordstyle=\color{pykeyword}\bfseries,
    stringstyle=\color{pystring},
    commentstyle=\color{gray},
    showstringspaces=false,
    numbers=left,
    numberstyle=\tiny,
    frame=single,
    breaklines=true,
    tabsize=4,
}

\numberwithin{equation}{section}
\theoremstyle{plain}
\newtheorem{theorem}{Theorem}
\newtheorem{defn}[theorem]{Definition}
\newtheorem{coro}[theorem]{Corollary}
\newtheorem{lemma}[theorem]{Lemma}
\newtheorem{prop}[theorem]{Proposition}
\newtheorem{obse}[theorem]{Observation}
\newtheorem{remark}[theorem]{Remark}
\begin{document}
\title{Mutual-Visibility of Tree and Its Line Graphs} 
\author{Tonny K B}
\address{Tonny K B, Department of Mathematics, College of Engineering Trivandrum, Thiruvananthapuram, Kerala, India, 695016.}
\email{tonnykbd@cet.ac.in}
\author{Shikhi M}
\address{Shikhi M, Department of Mathematics, College of Engineering Trivandrum, Thiruvananthapuram, Kerala, India, 695016.}
\email{shikhim@cet.ac.in}

\begin{abstract}
In this paper, we present a complete characterization of mutual-visibility sets in trees. It is shown that a subset $S$ is a mutual-visibility set of a tree $T$ if and only if it coincides with the set of leaves of the Steiner subtree $T\langle S\rangle$. For trees containing branch vertices, the notion of legs is introduced, and an explicit formula for the number of maximal mutual-visibility sets is derived in terms of the corresponding leg lengths. We prove that every tree is absolute-clear. It is further shown that, for every tree $T$ with at least two edges, the mutual-visibility number is preserved under the line graph operation, that is, $\mu(L(T))=\mu(T)$. Examples of unicyclic and block graphs for which this equality fails are also presented. Finally, a tight lower bound for the mutual-visibility number of the iterated line graph is established; namely, $\mu\bigl(L(L(T))\bigr)\ge \left\lfloor \frac{\Delta(T)^2}{3}\right\rfloor$.
\end{abstract} 
\maketitle
\section{Introduction}
The concept of mutual visibility in graphs originates from the idea of restricting shortest paths by excluding internal vertices from a specified set. Mutual-visibility sets, originally introduced as a tool for studying information flow and structural restrictions in complex networks, have since gained increasing attention due to their theoretical significance and diverse applications. For a simple graph $G(V,E)$ and a subset $X \subseteq V$, two vertices $u$ and $v$ are called $X$-visible \cite{sandi} if there exists a shortest path between them whose internal vertices lie outside $X$. When every pair of vertices in $X$ satisfies this condition, the set $X$ is referred to as a mutual-visibility set \cite{Stefano}. The maximum size of such a set is the mutual-visibility number $\mu(G)$, and the number of maximum mutual-visibility sets is denoted by $r_{\mu}(G)$. These definitions provide a unified framework for studying how shortest-path constraints shape the structure of subsets within a graph.

Interest in mutual visibility has increased notably in recent years, owing to its relevance in both theoretical studies and real-world applications. The earliest investigations can be traced back to the work of Wu and Rosenfeld on visibility issues in pebble graphs \cite{Geo_convex_1}, and a formal graph-theoretic treatment was later provided by Di Stefano \cite{Stefano}. Since then, mutual visibility has been used as a structural tool for understanding how information, influence, and coordination propagate in networks where interactions are limited by the underlying topology~\cite{MV_2, MV_5, MV_9}. Several variations and generalizations of the concept have subsequently emerged, broadening its scope and applicability~\cite{MV_10}. 

In this paper, we study mutual-visibility in trees. We show that a subset $S$ of vertices of a tree $T$ is a mutual-visibility set if and only if it coincides with the set of leaves of the Steiner subtree $T\langle S\rangle$. In~\cite{VP_2}, the notion of absolute-clear graphs was introduced. In this paper, it is shown that every tree is absolute-clear. In Section~\ref{P5.sec4}, it is shown that the mutual-visibility number of the line graph of a tree $T$ equals that of $T$. We also establish a sharp lower bound for the mutual-visibility number of the iterated line graph of a tree. .
\section{Notations and preliminaries}

Let $G = (V, E)$ represent a simple, undirected graph, where $V(G)$ denotes the set of vertices and $E(G)$ the set of edges. Throughout this work, we assume graphs are connected unless stated otherwise, meaning that there exists a path between every pair of vertices. The maximum degree of $G$ is denoted by $\Delta(G)$, or simply by $\Delta$ when no ambiguity arises. The \emph{line graph} $L(G)$ is the graph with vertex set $E$  and two vertices of $L(G)$ are adjacent if and only if the corresponding edges of $G$ are incident to a common vertex in $G$.

A sequence $(u_0,u_1,\ldots,u_n)$ is called a $(u_0,u_n)$\emph{-path} if $u_i u_{i+1}\in E(G)$ for all $0\le i<n$. A \emph{cycle} in $G$ is a path together with the edge $u_0u_n$. A connected graph with no cycles is called a \emph{tree}. In a tree $T$, a vertex of degree one is called a \emph{leaf}, and a vertex of degree at least three is called a \emph{branch vertex}. Let $\mathrm{Br}(T)=\{v\in V(T):\deg_T(v)\ge 3\}$ denote the set of branch vertices of $T$, and let $\mathcal{L}(T)$ denote the set of all leaves of $T$. For $u,v\in V(T)$, the unique $(u, v)$-path in $T$ is denoted by $P_T(u,v)$. A \emph{leg} of $T$ is the unique path joining a leaf vertex $u$ to its nearest branch vertex $b(u)$ in $T$, that is, the path $P_T(u,b(u))$.

Let $T$ be a tree and $S$ is non-empty subset of $V(T)$. The \emph{Steiner subtree} spanned by $S$, denoted $T\langle S\rangle$, is the unique connected subgraph of $T$ that contains all vertices of $S$ and is minimal by inclusion, meaning that no proper subgraph of $T\langle S\rangle$ remains connected and still contains $S$.  Equivalently, $T\langle S\rangle$ is the unique subtree of $T$ that contains $S$ with minimum edges.

Let $Q \subseteq V(G)$. A subset $W \subseteq V(G)\setminus Q$ is called a $c_Q$\emph{-visible set}\cite{VP_2} if it is $Q$-visible and, in addition, every pair $\{u,w\}$ with $u \in Q$ and $w \in W$ is $Q$-visible. A $c_Q$-visible set $W$ is said to be maximal if it is not properly contained in any larger $c_Q$-visible set. Moreover, $W$ is called an \emph{absolute $c_Q$-visible set}\cite{VP_2} if the set $Q$ itself forms a mutual--visibility set of $G$; equivalently, this holds when $Q \cup W$ is $Q$-visible. A subset $Q \subseteq V(G)$ is called \emph{disjoint-visible}\cite{VP_2} if, whenever there exist multiple maximal absolute $c_Q$-visible sets, they are mutually disjoint. A graph $G$ is said to be \emph{absolute-clear}\cite{VP_2} when every non-empty subset of $V(G)$ is disjoint-visible. A maximal $2$-connected subgraph of a graph $G$ is called a \emph{block} of $G$. A connected graph $G$ is called a \emph{block graph} if every block of $G$ is a clique.

For an integer $n\ge 2$, the \emph{triangular graph} $\mathcal{T}(n)$ is the graph whose
vertex set is $V(T(n))=\bigl\{\{i,j\} : 1\le i<j\le n\bigr\}$ and in which two vertices are adjacent if and only if the corresponding
2-subsets have non-empty intersection. 
The \emph{Tur\'an graph} $T_{n,m}$ is the complete $m$-partite graph on $n$ vertices
whose partite sets have sizes differing by at most one; in particular, it is
unique up to isomorphism. If $n=qm+s$, where $0\le s<m$, then $T_{n,m}$ has $s$
partite sets of size $q+1$ and $m-s$ partite sets of size $q$, and
\[
|E(T_{n,m})|
=
\left(1-\frac{1}{m}\right)\frac{n^2}{2}
-\frac{s(m-s)}{2m}.
\]
\emph{Tur\'an's theorem} states that among all simple graphs on $n$ vertices containing
no $K_{m+1}$, the graph $T_{n,m}$ has the maximum number of edges
(see, e.g., \cite{West}).
\section{Mutual-Visibility Sets of Trees}
\begin{lemma}\label{P5.lem2}
Let $T$ be a tree and let $S\subseteq V(T)$. 
Then $S$ is a mutual–visibility set of $T$ if and only if $S$ is precisely the
leaf set of the Steiner subtree $T\langle S\rangle$.
\end{lemma}
\begin{proof}
($\Rightarrow$) Suppose that $S$ is a mutual–visibility set of $T$. By definition of Steiner subtree $S \subseteq T\langle S\rangle$.
If some vertex $x\in S$ were not a leaf of $T\langle S\rangle$, then 
$\deg_{T\langle S\rangle}(x)\ge 2$. 
Then there exist distinct vertices $a,b\in S$ such that $x$ lies on the unique 
$(a, b)$-path $P_T(a,b)$ in $T$, contradicting the mutual-visibility of $S$. 
Hence every vertex of $S$ is a leaf of $T\langle S\rangle$, and therefore
$S\subseteq \mathcal{L}(T\langle S\rangle)$.

To prove the reverse inclusion, let $\ell\in \mathcal{L}(T\langle S\rangle)$.  
Since 
\[
T\langle S\rangle=\bigcup_{\{u,v\}\subseteq S} P_T(u,v),
\]
the unique edge incident with $\ell$ lies on some path $P_T(u,v)$ with 
$u,v\in S$. 
On such a path, a vertex of degree~1 in $T\langle S\rangle$ can only occur as an 
endpoint; hence $\ell\in\{u,v\}\subseteq S$.  
Thus $\mathcal{L}(T\langle S\rangle)\subseteq S$. 
Combining the two inclusions yields $S=\mathcal{L}(T\langle S\rangle)$.

\smallskip
($\Leftarrow$) Conversely, assume that 
$S=\mathcal{L}(T\langle S\rangle)$. 
Let $a,b\in S$ with $a\neq b$. 
Since $T\langle S\rangle$ is the union of all $P_T(u,v)$ with $u,v\in S$, the 
path $P_T(a,b)$ is contained in $T\langle S\rangle$.  
Because all vertices of $S$ are leaves of $T\langle S\rangle$, none of the 
internal vertices of $P_T(a,b)$ belongs to $S$.  
Hence $\{a,b\}$ is $S$–visible, and therefore $S$ is a mutual–visibility set of 
$T$.
\end{proof}
\begin{lemma}\label{P5.lem3}
Let $T(V,E)$ be a tree. Then,
$$ |\mathcal{L}(T)| = 2 + \sum_{v \in \rm{Br}(T)} (\deg_T(v)-2) $$
\end{lemma}
\begin{proof}
Let $V_1$ and  $V_2$ denote the sets of vertices of degree $1$ and  $2$ respectively.  
Write $n_1=|V_1|$, $n_2=|V_2|$, and $n_{\ge3}=|\rm{Br}(T)|$.  
By the Handshaking Lemma,
\[
\sum_{v\in V(T)} \deg_T(v) = 2|E(T)| = 2(|V(T)|-1).
\]
It follows that,
\begin{align*}
n_1 + 2n_2 + \sum_{v\in \mathrm{Br}(T)} \deg_T(v)
   &= 2(n_1 + n_2 + n_{\ge 3} - 1) \\
n_1 + \sum_{v\in \mathrm{Br}(T)} \deg_T(v)
   &= 2n_1 + 2n_{\ge 3} - 2 \\
n_1
   &= 2 + \sum_{v\in \mathrm{Br}(T)} \bigl(\deg_T(v) - 2\bigr).
\end{align*}

Since $n_1 = |\mathcal{L}(T)|$, the claim follows.
\end{proof}
\begin{obse}\label{P5.obs1}
Let $T$ be a tree and let $T'$ be a subtree of $T$. If $S\subseteq V(T)$ is a mutual-visibility set of $T$, then $S\cap V(T')$ is a mutual-visibility set of $T'$, since in a tree the unique path between any two vertices of $V(T')$ is the same in both $T$ and $T'$.
\end{obse}
\begin{lemma}\label{P5.lem6}
Let $T$ be a tree. If a mutual–visibility set $S\subseteq V(T)$ contains a vertex of the Steiner subtree spanned by $ \mathrm{Br}(T)$, then $|S| < |\mathcal{L}(T)|$.
\end{lemma}
\begin{proof}
Let $B= \mathrm{Br}(T)$ and let $S$ be a mutual–visibility set containing a vertex 
$b\in V(T\langle B \rangle)$.

\smallskip
\noindent\textbf{Case 1: $\boldsymbol{|B|=1}$.}
Let $B=\{c\}$ so that $c$ is the unique branch vertex of $T$. 
Then $T\langle B\rangle$ is a graph with only one vertex $c$ and hence $b=c$.  
Since $d=\deg_T(c)\ge3$, removing $c$ disconnects $T$ into $d$ components,
each of which is a path (a leg attached at $c$).  Denote these legs by
$L_1,\dots,L_d$.

Let $S$ be a mutual–visibility set with $c\in S$.  
If $x\in S\cap V(L_i)$ and $y\in S\cap V(L_j)$ for some $i\ne j$, then the
unique path $P_T(x,y)$ has $c$ as an internal vertex, contradicting mutual
visibility.  Thus all vertices of $S\setminus\{c\}$ lie in a single leg, say
$L_1$.

The induced subgraph on $V(L_1)\cup\{c\}$ is a path, and by
Observation~\ref{P5.obs1} the set $S\cap (V(L_1)\cup\{c\})=S$ is a mutual–visibility set in this path.
Since a path has exactly two leaves, any
mutual–visibility set on a path has size at most~$2$, that is, $|S|\leq 2$. On the other hand, all leaves of $T$ occur at the ends of the $d$ legs, so
$|\mathcal{L}(T)|=d\ge3$.  Hence $|S|<|\mathcal{L}(T)|$ in this case.

\smallskip
\noindent\textbf{Case 2: $\boldsymbol{|B|\ge2}$.}
We begin by proving that, after deleting $b$, all vertices of $S$ are contained in a single component of $T - b$.

Removing $b$ disconnects $T$ into $\deg_T(b) \geq 2$ components. If $x$ and $y$ lie in different components of $T-b$, then the unique $(x, y)$-path in $T$ passes through $b$ as an internal vertex. Therefore, $x$ and $y$ cannot both belong to $S$. Hence there exists a unique component $T_k$ of $T-b$ such that $S \subseteq V(T_k)\cup\{b\}$.

Let $U$ be the subtree of $T$ induced by $V(T_k)\cup\{b\}$. Then
$S \subseteq V(U)$ and, by Observation~\ref{P5.obs1}, $S$ is a mutual-visibility
set in $U$. Since $U$ is a tree, it follows that
\begin{equation}\label{P5.eq4}
|S| \le \mu(U) = |\mathcal{L}(U)|.
\end{equation}

Note that any leaf of $U$ different from $b$ is also a leaf of $T$. Thus $\mathcal{L}(U)\setminus\{b\} \;\subseteq\; \mathcal{L}(T)$, and so
\begin{equation}\label{P5.eq2}
   |\mathcal{L}(U)| \;\le\; |\mathcal{L}(T)\cap V(U)| + 1.
\end{equation}

We next show that at least two leaves of $T$ lie outside $U$. Since $|B|\ge2$ and $b\in V(T\langle B \rangle)$, the vertex $b$ lies on the path joining at least two branch vertices. 

If $b$ is an internal vertex of $T\langle B \rangle$, then there exists a component $T_j\ne T_k$ of $T-b$ that contains a branch 
vertex of $T$. Consider the component $T_j$ as a tree in its own right.  Since it has a branch vertex, Lemma~\ref{P5.lem3} implies that $T_j$ has at 
least three leaves. When we reattach $b$ to $T_j$ to obtain $T$, at most one of these leaves 
(namely, the neighbor of $b$ in $T_j$) can lose its leaf status; hence at least 
two leaves of $T_j$ remain leaves of $T$. 
All such leaves lie in $T_j$, which is disjoint from $U$, so $|\mathcal{L}(T)\setminus V(U)|  \geq 2$.

If $b$ is a branch vertex of $T$, then $d=\deg_T(b)\ge3$, and so $T-b$ has at least three components, say $T_1,\dots,T_d$, one of which is $T_k$.
For each $j\ne k$, view $T_j$ as a tree in its own right. Since every tree has at least one leaf, each $T_j$ contains a leaf. Let $x_j$ be the neighbour of $b$ in $T_j$. When we reattach $b$ to $T_j$ to obtain $T$, the only vertex whose degree changes is $x_j$; hence at most one leaf of $T_j$ can lose its leaf status in
$T$. Therefore each component $T_j$ with $j\ne k$ contributes at least one leaf of $T$, and all such leaves lie outside $U$. Since there are at least two such components $T_j$ (because $d\ge3$ and we
exclude $T_k$), we again obtain $|\mathcal{L}(T)\setminus V(U)|  \geq 2$.

\smallskip
In either subcase we have $|\mathcal{L}(T)\setminus V(U)|\ge2$. Therefore
\[
   |\mathcal{L}(T)| 
   \;=\; |\mathcal{L}(T)\cap V(U)| + |\mathcal{L}(T)\setminus V(U)|
   \;\ge\; |\mathcal{L}(T)\cap V(U)| + 2,
\]
and hence $|\mathcal{L}(T)\cap V(U)| \;\le\; |\mathcal{L}(T)| - 2$.
Combining this with \eqref{P5.eq2} gives
\[
   |\mathcal{L}(U)| 
   \;\le\; |\mathcal{L}(T)\cap V(U)| + 1
   \;\le\; (|\mathcal{L}(T)| - 2) + 1
   \;=\; |\mathcal{L}(T)| - 1.
\]
Finally, combining this with \eqref{P5.eq4}, we obtain $|S| \;\le\; |\mathcal{L}(U)| \;\le\; |\mathcal{L}(T)| - 1
   \;<\; |\mathcal{L}(T)|$, as required.

\smallskip
In both cases, whenever $S$ contains a vertex $b\in V(T\langle B\rangle)$, we
have $|S|<|\mathcal{L}(T)|$. 
This completes the proof.
\end{proof}
If $T$ is a path, then Bujtás et al.~\cite{sandi} showed that $r_\mu(T)=\binom{|V(T)|}{2}$. We next generalize this result to all trees.
\begin{lemma}[{\cite[Coro 4.3]{Stefano}}]\label{P5.th1}
Let $T = (V, E )$ be a tree and $L$ the set of its leaves. Then $\mathcal{L}$ is a mutual-visibility set and $\mu(T ) = |\mathcal{L}|$.     
 \end{lemma}
\begin{theorem}\label{P5.th3}
Let $T$ be a tree containing at least one branch vertex. Let $\ell(u)$ denote the length of the leg
of $T$ determined by the leaf vertex $u$ of $T$. 
Then 
\[
   r_{\mu}(T) \;=\; \prod_{u\in \mathcal{L}(T)} \ell(u).
\]
\end{theorem}
\begin{proof}
Let $T$ be a tree with at least one branch vertex, and let
$\mathcal{L}(T)$ denote its leaf set.  
By Lemma~\ref{P5.th1}, we have $
   \mu(T) = |\mathcal{L}(T)|$.

Let $B$ be the set of branch vertices of $T$, and let $H=T\langle B\rangle$
be the Steiner subtree spanned by $B$.  
By Lemma~\ref{P5.lem6}, no mutual–visibility set of cardinality
$\mu(T)$ can contain a vertex of $H$.  
Thus for every $S$ with $|S|=\mu(T)$ we have $S \subseteq V(T)\setminus V(H)$.

For each leaf $u\in\mathcal{L}(T)$, let $b(u)$ be its nearest branch vertex
and let $L(u)$ be the leg of $T$ determined by $u$, that is, the path
$P_T(u,b(u))$.  
By the definition of $b(u)$, every internal vertex of $L(u)$ has degree~$2$ in
$T$, and $L(u)$ meets $H$ only in the vertex $b(u)$.  
It follows that the sets $W(u) = V(L(u))\setminus\{b(u)\}$ where $u\in\mathcal{L}(T)$,
are pairwise disjoint and
\[
   V(T)\setminus V(H) = \biguplus_{u\in\mathcal{L}(T)} W(u).
\]
In particular, any set $S$ of cardinality $\mu(T)$ satisfies
\[
   S = \biguplus_{u\in\mathcal{L}(T)} S_u,
   \qquad\text{where } S_u=S\cap W(u).
\]

Since $|S|=\mu(T)=|\mathcal{L}(T)|$, we have
\begin{equation}\label{P5.eq3}   \sum_{u\in\mathcal{L}(T)} |S_u| = |\mathcal{L}(T)|.
\end{equation}

\smallskip
We claim that each set $W(u)$ contributes exactly one vertex to $S$.  
For $u\in\mathcal{L}(T)$, we have
$S_u \;=\; S\cap W(u) \;=\; S\cap V(L(u))$. By Observation~\ref{P5.obs1}, the set $S_u$ is a mutual--visibility set of the path
$L(u)$.  
Since a path has exactly two leaves, Lemma~\ref{P5.th1} applied to $L(u)$
implies that $|S_u|\le 2$ for all $u\in\mathcal{L}(T)$.

We now show that in fact $|S_u|\le 1$ for every $u$.  
Suppose, to the contrary, that there exists $u_0\in\mathcal{L}(T)$ with
$|S_{u_0}|=2$.  
Choose distinct vertices $x_1,x_2\in S_{u_0}$, ordered so that $x_2$ is
closer to $b(u_0)$ than $x_1$ along the leg $L(u_0)$.  
Since $T$ has at least one branch vertex, it has at least three leaves, and
hence $|S|=\mu(T)=|\mathcal{L}(T)|\ge 3$.  
Thus there exists a leaf $v\ne u_0$ with $S_v\ne\emptyset$, and we may take
$y\in S_v$.  
The unique path $P_T(x_1,y)$ then contains $x_2$ as an internal vertex,
contradicting the $S$--visibility of $x_1$ and $y$.  
Therefore $|S_u|\le 1$ for every $u\in\mathcal{L}(T)$.

Combining this with \eqref{P5.eq3} yields $   |S_u| = 1$ for all $u\in\mathcal{L}(T)$. Thus every maximum mutual--visibility set $S$ contains exactly one vertex from each $W(u)$.

Conversely, we claim that any choice of one vertex from each set $W(u)$, $u\in\mathcal{L}(T)$, yields a mutual–visibility set $S$ of cardinality $\mu(T)$.  Let $x,y\in S$ be distinct, and suppose $x\in W(u)$ and $y\in
W(v)$.  The unique path $P_T(x,y)$ decomposes into three segments: the path from $x$ to $b(u)$, the subpath in $H$ from $b(u)$ to $b(v)$, and the path from $b(v)$ to $y$.  All internal vertices of the first and last segments lie in $W(u)\setminus\{x\}$ and $W(v)\setminus\{y\}$, respectively, and these sets
are disjoint from $S$, since $S$ contains exactly one vertex from each leg. The middle segment lies entirely in $H$, which contains no vertex of $S$. Therefore no internal vertex of $P_T(x,y)$ belongs to $S$, and hence
$\{x,y\}$ is $S$–visible.  As $x$ and $y$ were arbitrary, it follows that $S$ is a mutual–visibility set.

For each leaf $u\in\mathcal{L}(T)$ there are exactly $\ell(u)$ admissible choices for the unique vertex selected from $W(u)$, and the choices for
distinct leaves are independent in forming a mutual–visibility set of cardinality $\mu(T)$.  Consequently, the total number of such sets is
\[
   r_{\mu}(T) = \prod_{u\in\mathcal{L}(T)} \ell(u).
\]

\end{proof}

\begin{defn}
Let $T$ be a tree, let $S\subseteq V(T)$, and let $x\in V(T)$. The attachment point of $x$ with respect to $S$ in $T$ is defined as the vertex of the Steiner subtree $T\langle S\rangle$ that is closest to $x$, and is denoted by $a_S(x)$. Note that $a_S(x)$ is unique.
\end{defn}
\begin{lemma}\label{P5.lem5}
Let $T$ be a tree, let $Q\subseteq V(T)$ be a mutual–visibility set of $T$, and let $u\in Q$ and $w\in V(T)\setminus Q$. 
Then $\{u,w\}$ is $Q$–visible if and only if $a_Q(w)\notin Q\setminus\{u\}$, 
where $a_Q(w)$ denotes the attachment point of $w$ with respect to $Q$ in $T$.
\end{lemma}

\begin{proof}
Let $H=T\langle Q\rangle$ be the Steiner subtree spanned by $Q$. 
Since $Q$ is a mutual–visibility set of $T$, we have $Q=\mathcal{L}(H)$ by Lemma~\ref{P5.lem2}.

\smallskip
Consider vertices $u\in Q$ and $w\in V(T)\setminus Q$. 
Because $T$ is a tree, there exists a unique path $P_T(u,w)$ between them. 
By the definition of the attachment point $a_Q(w)$, the intersection of $P_T(u,w)$ with $H$ is precisely the path in $H$ joining $u$ and $a_Q(w)$; that is,
\begin{equation}\label{P5.eq1}
V(P_T(u,w))\cap V(H) = V\bigl(P_H(u,a_Q(w))\bigr).
\end{equation}
The internal vertices of $P_H(u,a_Q(w))$ are not leaves of $H$, and hence not elements of $Q=\mathcal{L}(H)$. Thus, if a vertex of $Q$ appears as an internal vertex of $P_T(u,w)$, then $a_Q(w)\in Q$. That is, if $a_Q(w) \notin Q$ ($a_Q(w) \in V(H)\setminus Q$, since $a_Q(w) \in V(H)$), then $u$ and $w$ are $Q$-visible. If $a_Q(w)\in Q\setminus\{u\}$, then $a_Q(w)$ is an internal vertex of $P_T(u,w)$, and hence $\{u,w\}$ is not $Q$–visible.  If $a_Q(w)=u$, then by~\eqref{P5.eq1} we have $V(P_T(u,w))\cap V(H)=V\bigl(P_H(u,u)\bigr)=\{u\}$, so no internal vertex of the path lies in $Q$, and $\{u,w\}$ is $Q$–visible. Therefore, $\{u,w\}$ is $Q$–visible if and only if either $a_Q(w)=u$ or $a_Q(w)\in V(H)\setminus Q$.
\end{proof}

\begin{prop}\label{P5.prop1}
Let $T$ be a tree, let $Q\subseteq V(T)$ be a nonempty subset, and let $H=T\langle Q\rangle$ be the Steiner subtree spanned by $Q$. 
For each vertex $x\in V(H)$, define
\[
B(x) = \{ w\in V(T)\setminus Q:\ a_Q(w) = x \},
\]
where $a_Q(w)$ denotes the attachment point of $w$ with respect to $Q$. 
Then the family $\{B(x) \neq \emptyset:x\in V(H)\}$ forms a partition of $V(T)\setminus Q$.
\end{prop}
\begin{proof}
Define a relation $\sim$ on $V(T)\setminus Q$ by
\[
w_1\sim w_2 \quad \text{if and only if} \quad a_Q(w_1)=a_Q(w_2).
\]

\smallskip
\noindent \emph{Reflexivity}: For any $w\in V(T)\setminus Q$, we have $a_Q(w)=a_Q(w)$, hence $w\sim w$.

\smallskip
\noindent \emph{Symmetry}: If $w_1\sim w_2$, then $a_Q(w_1)=a_Q(w_2)$, and consequently $a_Q(w_2)=a_Q(w_1)$; thus $w_2\sim w_1$.

\smallskip
\noindent\emph{Transitivity}: If $w_1\sim w_2$ and $w_2\sim w_3$, then $a_Q(w_1)=a_Q(w_2)=a_Q(w_3)$, which implies $w_1\sim w_3$.

\smallskip
Hence $\sim$ is an equivalence relation on $V(T)\setminus Q$. 
By definition of $\sim$, the equivalence class of any vertex $w\in V(T)\setminus Q$ is
\[
[w] \;=\; \{\,z\in V(T)\setminus Q:\ a_Q(z)=a_Q(w)\,\}
     \;=\; B\bigl(a_Q(w)\bigr).
\]
Since $a_Q(w)\in V(H)$, each equivalence class $[w]$ coincides with $B(x)$ for some $x\in V(H)$. 
Conversely, for any nonempty $B(x)$ with $x\in V(H)$,
\[
B(x)=  \{ w\in V(T)\setminus Q:\ a_Q(w)=x\}=    \{w\in V(T)\setminus Q:\ a_Q(w)=a_Q(x)\}
= [x].
\]

\smallskip
Therefore, the family $\{B(x)\ne\emptyset : x\in V(H)\}$ consists of the distinct equivalence classes of $\sim$ and thus forms a partition of $V(T)\setminus Q$.
\end{proof}
\begin{theorem}\label{P5.th2}
Every tree is absolute–clear.
\end{theorem}
\begin{proof}
Let $T$ be a tree and fix a nonempty subset $Q\subseteq V(T)$. 
If $Q$ is not a mutual–visibility set of $T$, then by definition, no absolute $c_Q$–visible set exists, and the disjoint–visible condition is vacuously satisfied. 
Hence we may assume throughout that $Q$ is mutual–visible.

\smallskip
\noindent\textbf{Case 1: $\boldsymbol{|Q|=1}$.}
Let $Q=\{u\}$. Then $Q$ is trivially a mutual–visibility set. 
Since the Steiner subtree $T\langle Q\rangle$ consists of the single vertex $u$, we have $a_Q(w)=u$ for every $w\in V(T)\setminus Q$. 
By Lemma~\ref{P5.lem5}, each pair $\{u,w\}$ is therefore $Q$–visible.

For $w_1,w_2\in V(T)\setminus Q$, the pair $\{w_1,w_2\}$ is $Q$–visible if and only if $u$ is not an internal vertex of the unique path $P_T(w_1,w_2)$; equivalently, $w_1$ and $w_2$ lie in the same component of $T-u$. 
Thus a subset $W\subseteq V(T)\setminus Q$ is an absolute $c_Q$–visible set if and only if $W$ is contained in a vertex set of single component of $T-u$, and the maximal such sets are exactly the vertex sets of the components of $T-u$. 
These maximal sets are clearly pairwise disjoint.

\smallskip
\noindent\textbf{Case 2: $\boldsymbol{|Q|\ge2}$.}
Let $H=T\langle Q\rangle$ be the Steiner subtree spanned by $Q$.  
Since $Q$ is mutual–visible, Lemma~\ref{P5.lem2} implies that 
\[
Q = \mathcal{L}(H),
\]
the set of leaves of $H$.  
For each $x\in V(H)$ define
\[
B(x) = \{\, w\in V(T)\setminus Q : a_Q(w)=x \,\}.
\]
By Proposition~\ref{P5.prop1}, the nonempty sets $B(x)$, $x\in V(H)$, form a partition of $V(T)\setminus Q$.

If $w\in B(x)$ for some $x\in Q$ and $u\in Q\setminus\{x\}$, then $a_Q(w)=x\ne u$, and Lemma~\ref{P5.lem5} shows that $\{u,w\}$ is not $Q$–visible.  
Consequently, no absolute $c_Q$–visible set may contain any vertex from a branch $B(x)$ with $x\in Q$.  
Hence every absolute $c_Q$–visible set $W$ satisfies
\[
W \subseteq \bigcup_{x\in V(H)\setminus Q} B(x).
\]
Therefore, if $\bigcup_{x\in V(H)\setminus Q} B(x)$ itself is a $c_Q$–visible set, then it is the unique maximal absolute $c_Q$--visible set . We claim that $W=\bigcup_{x\in V(H)\setminus Q}$ is $c_Q$–visible.

\smallskip
We first verify that $\{q,w\}$  is $Q$-visible where $q\in Q$ and $w\in W$. Note that $a_Q(w)\in V(H)\setminus Q$ is an internal vertex of $H$. The path $P_T(q,w)$ decomposes into a subpath from $q$ to $a_Q(w)$ within $H$, followed by a subpath from $a_Q(w)$ to $w$ outside $H$. Since $q\in Q=\mathcal{L}(H)$ is a leaf of $H$, the unique subpath from $q$ to $a_Q(w)$ has $q$ as its only vertex in $Q$, and all of its internal vertices lie in $V(H)\setminus Q$.  
The remaining subpath from $a_Q(w)$ to $w$ lies entirely in $V(T)\setminus V(H)$ and therefore contains no vertex of $Q$. Consequently, no internal vertex of $P_T(q,w)$ belongs to $Q$, and hence
$\{q,w\}$ is $Q$--visible.

\smallskip
We now verify that any two vertices of $W$ are also $Q$--visible. Let $w_1,w_2\in W$, and put $x_i=a_Q(w_i)\in V(H)\setminus Q$ for $i=1,2$.

\smallskip
\textbf{Subcase 2.1: $\boldsymbol{x_1 = x_2}$.}
Write $x_1=x_2=x$.  
Then $w_1,w_2\in B(x)$.  
Let $T_x$ denote the subtree of $T$ induced by $B(x)$.  
By the definition of the attachment point, $T_x$ meets $H$ only in the vertex $x$.  
Since $Q\subseteq V(H)$ and $x\notin Q$, we have $T_x\cap Q=\emptyset$.  
Because $T_x$ is a tree, the entire path $P_T(w_1,w_2)$ lies in $T_x$, and hence contains no internal vertex of $Q$.  
Thus $\{w_1,w_2\}$ is $Q$–visible.

\smallskip
\textbf{Subcase 2.2: $\boldsymbol{x_1 \neq x_2}$.}
The path $P_T(w_1,w_2)$ is the concatenation of the paths $P_T(w_1,x_1)$, $P_T(x_1,x_2)$ inside $H$, and $P_T(x_2,w_2)$.  
All internal vertices of the paths $P_T(w_i,x_i)$ lie in $V(T)\setminus V(H)$ and thus avoid $Q$.  
The middle subpath $P_T(x_1,x_2)\subseteq H$ has endpoints $x_1,x_2$, which are internal vertices of $H$.  
A leaf of $H$ (i.e., a vertex of $Q$) cannot occur as an internal vertex of this path.  
Hence no internal vertex of $P_T(w_1,w_2)$ lies in $Q$, and $\{w_1,w_2\}$ is $Q$–visible.

\smallskip
We have shown that $Q\cup W$ is $Q$--visible, and therefore $W$ is the unique maximal absolute $c_Q$–visible set when $|Q|\geq 2$.
 
\smallskip
Combining the two cases, we conclude that for every nonempty $Q\subseteq V(T)$, the maximal absolute $c_Q$–visible sets are pairwise disjoint.  
Hence every tree $T$ is absolute–clear.
\end{proof}
In \cite{VP_1}, an exponential-time algorithm was given for computing the visibility polynomial of a graph $G$. In \cite{VP_2}, it was shown that the visibility polynomial of the corona product $G \odot H$ can be computed efficiently whenever $G$ is absolute-clear. Since every tree is absolute-clear, it follows that the visibility polynomial of the corona product of a tree with an arbitrary graph can be computed efficiently.

\section{Mutual-Visibility of Line Graphs and Iterated Line Graphs of Trees}\label{P5.sec4}
\begin{lemma}[{\cite[Th.~4.2]{Stefano}}]\label{P5.lem8}
Let $G=(V,E)$ be a connected block graph, and let $X$ be the set of all
cut--vertices of $G$.  
Then $V\setminus X$ is a mutual--visibility set of $G$, and $  \mu(G)=|V\setminus X|$.
\end{lemma}
\begin{theorem}\label{P5.th4}
Let $T$ be a tree with at least two edges.  
Then $ \mu\bigl(L(T)\bigr) \;=\; \mu(T)$.
\end{theorem}

\begin{proof}
By Lemma~\ref{P5.th1}, the mutual-visibility number of a tree is the number of its leaves. Let $G=L(T)$ be the line graph of $T$.  
In \cite{Harary} it is shown that $G$ is a connected block graph in which every cut--vertex lies in exactly two blocks. Let $X$ denote the set of cut--vertices of $G$.

Since $G$ is a connected block graph, Lemma~\ref{P5.lem8}, applies and yields $ \mu(G)=|V(G)\setminus X|$. Thus it remains to identify the set $V(G)\setminus X$. Let the vertex $v_e$ of $G$ corresponds to an edge $e=uv$ of $T$. Such a vertex is a cut--vertex of $G$ precisely when both $u$ and $v$ are non--leaf vertices of $T$. Indeed, if $e=uv$ with both $u$ and $v$ non-leaves, then in $T-e$ the vertices $u$ and $v$ lie in distinct components, and every edge incident with $u$ lies in one block of $L(T)$ while every edge incident with $v$ lies in another; hence the vertex $v_e$ of $L(T)$ is a cut–vertex. Conversely, if $e$ is incident with a leaf of $T$, then $v_e$ lies in exactly 
one clique of $L(T)$ and cannot be a cut–vertex.  
Equivalently, $v_e\notin X$ if and only if $e$ is incident with a leaf of $T$. Hence $V(G)\setminus X
   =
   \{ v_e \in V(G) : e\in E(T) \text{ is incident with a leaf of }T \,\}$.

Since $T$ has at least two edges, each leaf of $T$ is incident with a unique edge, and distinct leaves correspond to distinct such edges. Therefore $
   |V(G)\setminus X| = |\mathcal{L}(T)|$. It follows that, $
   \mu\bigl(L(T)\bigr)
   = |V(G)\setminus X|
   = |\mathcal{L}(T)|
   = \mu(T)$,
which completes the proof.
\end{proof}

The equality $\mu(L(T))=\mu(T)$ holds for trees, and it also holds for cycles,
since $L(C_n) = C_n$. However, it does not extend to all unicyclic graphs.
\begin{prop}\label{P5.prop3}
There exists a unicyclic graph $G$ such that $\mu(L(G))\neq \mu(G)$.
\end{prop}

\begin{proof}
Let $G$ be obtained from $K_3$ with vertices $a,b,c$ by attaching a leaf to each of $a,b,c$, denoted by $a',b',c'$, respectively. Then $G$ is unicyclic. The set $\{a',b',c'\}$ is a mutual-visibility set, so $\mu(G)\geq 3$. Suppose $X$ is a mutual-visibility set with $|X|=4$. Since $G$ has only three non-leaf vertices, $X$ contains at least one leaf. If $X$ contains exactly one leaf, say $a'$, then $X=\{a',a,b,c\}$, but the unique shortest $(a',b)$-path $(a',a,b)$ has internal vertex $a\in X$, a contradiction. If $X$ contains exactly two leaves, say $a',b'$, then $X$ contains two vertices from $\{a,b,c\}$, and hence at least one of $a$ or $b$ lies in $X$; however, the unique shortest $(a',b')$-path $(a',a,b,b')$ has an internal vertex in $X$, a contradiction. If $X$ contains all three leaves, then $X$ contains one of $a,b,c$, say $a$, but the unique shortest $(a',b')$-path $(a',a,b,b')$ again has an internal vertex in $X$, a contradiction. Thus $\mu(G)=3$.

Viewing the edges of $G$ as vertices of $L(G)$, the set $\{aa',bb',ca,bc\}$ is a mutual-visibility set of $L(G)$. Hence $\mu(L(G))\ge 4>3=\mu(G)$, and therefore $\mu(L(G))\neq \mu(G)$.
\end{proof}
Thus, trees and cycles preserve the mutual-visibility number under the line graph operator, while Proposition~\ref{P5.prop3} shows that this property fails even within the class of unicyclic graphs. Moreover, for block graphs this preservation need not hold; for example, $\mu(L(K_n))=\left\lfloor \frac{n^2}{3}\right\rfloor > \mu(K_n)=n$ (see Proposition~\ref{P5.prop2}). This motivates the following problems.

\smallskip
\noindent
\textbf{Problem 1.} Characterize the connected unicyclic graphs $G$ for which $\mu(L(G))=\mu(G)$.

\smallskip
\noindent
\textbf{Problem 2.} Characterize the connected graphs $G$ for which $\mu(L(G))=\mu(G)$.

\smallskip
Theorem~\ref{P5.th4} further raises the question whether the identity $\mu(L(T))=\mu(T)$ extends to the iterated line graph $L(L(T))$ for every tree $T$. We show that this does not hold in general.
\begin{prop}\label{P5.prop2}
For every integer $n\ge 2$,
\[
\mu\bigl(L(L(K_{1,n}))\bigr)
=\mu(\mathcal{T}(n))=
\left\lfloor \frac{n^2}{3}\right\rfloor .
\]
\end{prop}
\begin{proof}
Since the line graph of a star is complete, we have $L(K_{1,n}) \cong K_n$.
Therefore $L(L(K_{1,n})) \cong L(K_n)$, and the latter is the triangular graph $\mathcal{T}(n)$.

Recall that the vertices of $\mathcal{T}(n)$ are the $2$-subsets of 
$S=\{1,2,\dots,n\}$, where two vertices are adjacent if and only if the
corresponding $2$-subsets intersect. Let $X \subseteq V(\mathcal{T}(n))$. Identifying each vertex of $\mathcal{T}(n)$ with the corresponding edge of $K_n$, define a graph $H_X$ with vertex set $S$ and 
edge set $E(H_X)=X$.

We claim that $X$ is a mutual-visibility set in $\mathcal{T}(n)$ if and only if $H_X$ is $K_4$-free.

Assume that $X$ is a mutual-visibility set in $\mathcal{T}(n)$. Suppose, for a contradiction,
that $H_X$ contains a $K_4$ on vertices $a,b,c,d \in S$. Then the vertices of $\mathcal{T}(n)$ corresponding to the edges $ab, ac, ad, bc, bd, cd$ all belong to $X$. In $\mathcal{T}(n)$, the vertices $ab$ and $cd$
are non-adjacent, and their common neighbours are precisely $ac, ad, bc,$ and $bd$.
Thus every shortest $(ab,cd)$-path has its internal vertex in $X$, contradicting
the assumption that $X$ is a mutual-visibility set. Hence $H_X$ is $K_4$-free.

Conversely, assume that $H_X$ is $K_4$-free, and let $e,f \in X$ be distinct.
If the corresponding edges in $K_n$ intersect, then $e$ and $f$ are adjacent in
$\mathcal{T}(n)$, and hence are trivially mutually visible. Now suppose that $e$ and $f$ correspond to disjoint edges $ab$ and $cd$, where
$a,b,c,d$ are distinct. Since $H_X$ is $K_4$-free, not all of the vertices $ac, ad, bc, bd$ belong to $X$. Choose one such vertex, say $ac \notin X$. Then $ac$ is adjacent to both $e$ and $f$, so $(e,ac,f)$ is a shortest path between $e$ and $f$ whose internal
vertex does not belong to $X$. Therefore $e$ and $f$ are mutually visible.

It follows that $\mu(\mathcal{T}(n)) = \max\{\, |E(H)| : H \text{ is a $K_4$-free graph on } S \,\}$.
By Tur\'an's theorem, the maximum number of edges in a $K_4$-free graph on $n$
vertices is attained by the Tur\'an graph $T_{n,3}$. Hence
\[
\mu(\mathcal{T}(n)) = |E(T_{n,3})| =  \frac{n^2}{3} - \frac{3s - s^2}{6}.
\]
where $n = 3q + s$ with $s \in \{0,1,2\}$. Since $0 \le \frac{3s - s^2}{6} < 1$ for $s \in \{0,1,2\}$, we obtain $\mu(\mathcal{T}(n)) = \left\lfloor \frac{n^2}{3} \right\rfloor$.
\end{proof}
\begin{coro}
For every $n \ge 4$, $\mu\bigl(L(L(K_{1,n}))\bigr) > \mu(K_{1,n})$, since 
$\mu(K_{1,n}) = n$ (see \cite{Stefano}). In particular, the equality 
$\mu(L(T)) = \mu(T)$ does not extend to $L(L(T))$ for all trees.
\end{coro}

Next, we establish a sharp lower bound for $\mu\bigl(L(L(T))\bigr)$.
\begin{lemma}[{\cite[Lem 2.2]{Stefano}}]\label{P5.lem9}
Let $H$ be a convex subgraph of a graph $G$. Then $\mu(H) \leq \mu(G)$.
\end{lemma}
\begin{theorem}\label{P5.th5}
Let $T$ be a tree with at least two edges. Then
\[
\mu\bigl(L(L(T))\bigr)\ge \left\lfloor \frac{\Delta(T)^2}{3}\right\rfloor .
\]
Moreover, this bound is sharp.
\end{theorem}

\begin{proof}
Let $\Delta = \Delta(T)$, and let $x \in V(T)$ be a vertex of degree $\Delta$.
Let $E_x=\{e_1,e_2,\dots,e_{\Delta}\}$ be the set of edges of $T$ incident with $x$.

The vertices of $L(L(T))$ correspond to pairs of adjacent edges of $T$, that is,
to length-$2$ paths in $T$. Consider the induced subgraph $H$ of $L(L(T))$ whose
vertices correspond to all length-$2$ paths in $T$ with middle vertex $x$. Thus $V(H)=\{\,v_{ij} : 1 \le i < j \le \Delta\,\}$, where $v_{ij}$ denotes the vertex of $L(L(T))$ corresponding to the path formed
by the adjacent edges $e_i$ and $e_j$.

We show that $H\cong \mathcal{T}(\Delta)$, the triangular graph. Recall that the vertices of $\mathcal{T}(\Delta)$ are the
$2$-subsets of $\{1,2,\dots,\Delta\}$, and two vertices are adjacent if and only
if the corresponding $2$-subsets intersect. Define
\[
\varphi:V(H)\longrightarrow V(\mathcal{T}(\Delta)),\qquad
\varphi(v_{ij})=\{i,j\}.
\]
Clearly, $\varphi$ is a bijection. Now $v_{ij}$ and $v_{k\ell}$ are adjacent in
$H$ if and only if the corresponding length-$2$ paths in $T$ share an edge,
which happens if and only if $ \{i,j\}\cap\{k,\ell\}\neq \emptyset$. By the definition of $\mathcal{T}(\Delta)$, this is equivalent to saying that
$\varphi(v_{ij})$ and $\varphi(v_{k\ell})$ are adjacent in $\mathcal{T}(\Delta)$. Hence
$\varphi$ is an isomorphism, and therefore $
H\cong \mathcal{T}(\Delta)$.

Next we show that $H$ is convex in $L(L(T))$. Let $v_{ij}, v_{k\ell} \in V(H)$.
If $v_{ij}$ and $v_{k\ell}$ are adjacent in $L(L(T))$, then every shortest
$(v_{ij}, v_{k\ell})$-path lies entirely in $H$. Suppose now that $v_{ij}$ and $v_{k\ell}$ are not adjacent in $L(L(T))$. Then
$\{i,j\} \cap \{k,\ell\} = \emptyset$, so the indices $i,j,k,\ell$ are distinct.
Since all the edges $e_i, e_j, e_k, e_\ell$ are incident with $x$, the vertex
$v_{ik}$ belongs to $H$ and is adjacent to both $v_{ij}$ and $v_{k\ell}$.
Thus $d_{L(L(T))}(v_{ij}, v_{k\ell}) = 2$.

Now let $w$ be an internal vertex of a shortest $(v_{ij}, v_{k\ell})$-path in
$L(L(T))$. Then $w$ is adjacent to both $v_{ij}$ and $v_{k\ell}$, and hence the
length-$2$ path in $T$ corresponding to $w$ shares one edge with the path
corresponding to $v_{ij}$ and one edge with the path corresponding to $v_{k\ell}$.
Therefore, this path uses one edge from $\{e_i, e_j\}$ and one edge from
$\{e_k, e_\ell\}$. Since all these edges are incident with $x$, the two edges
corresponding to $w$ are also incident with $x$. Hence $w$ corresponds to a
length-$2$ path in $T$ with middle vertex $x$, and so $w \in V(H)$.

Thus every shortest path in $L(L(T))$ joining two vertices of $H$ lies entirely
in $H$, and therefore $H$ is convex in $L(L(T))$. By Lemma~\ref{P5.lem9}, it follows that $\mu(H)\le \mu\bigl(L(L(T))\bigr)$.
Using the isomorphism $H\cong \mathcal{T}(\Delta)$ and the fact that
\[
\mu(\mathcal{T}(\Delta))=\left\lfloor \frac{\Delta^2}{3}\right\rfloor,
\]
we obtain
\[
\mu\bigl(L(L(T))\bigr)\ge \mu(H)=\mu(\mathcal{T}(\Delta))
=\left\lfloor \frac{\Delta^2}{3}\right\rfloor.
\]

Finally, the bound is sharp. Indeed, if $T=K_{1,n}$, then $\Delta(T)=n$, and
Proposition~\ref{P5.prop2} gives
\[
\mu\bigl(L(L(T))\bigr)
=
\left\lfloor \frac{n^2}{3}\right\rfloor
=
\left\lfloor \frac{\Delta(T)^2}{3}\right\rfloor .
\]
This completes the proof.
\end{proof}
\begin{remark}
Theorem~\ref{P5.th5} shows that the behavior of $\mu(L(L(T)))$ is governed not only by the leaf structure of $T$, but also by its local branching. In particular, in contrast to the identity $\mu(L(T))=\mu(T)$, the parameter $\mu(L(L(T)))$ may grow quadratically in the maximum degree of $T$.
\end{remark}
For $n \ge 4$, we have $\mu\bigl(L(L(P_n))\bigr)=2$, since 
$L(L(P_n)) \cong L(P_{n-1}) \cong P_{n-2}$ is a path, and hence its
mutual-visibility number is $2$ (see \cite{Stefano}). Theorem~\ref{P5.th5} shows that the lower bound $
\mu\bigl(L(L(T))\bigr)\ge \left\lfloor \frac{\Delta(T)^2}{3}\right\rfloor$
is attained for stars. However, this bound is not attained for paths, since
$\left\lfloor \frac{\Delta(P_n)^2}{3}\right\rfloor = 1$. This naturally leads
to the following problem.

\smallskip
\noindent
\textbf{Problem 3.} Characterize all trees $T$ satisfying
\[
\mu\bigl(L(L(T))\bigr)=\left\lfloor \frac{\Delta(T)^2}{3}\right\rfloor.
\]
\section{Conclusion}
This paper provides a comprehensive study of mutual visibility in trees. A structural characterization of mutual-visibility sets is obtained by showing that a subset $S\subseteq V(T)$ is a mutual-visibility set if and only if it coincides with the set of leaves of the Steiner subtree $T\langle S\rangle$. For trees containing branch vertices, the notion of legs is introduced, and an explicit expression for the number of maximal mutual-visibility sets is derived in terms of the corresponding leg lengths. It is shown that the mutual-visibility number is preserved under the line graph operation for trees with at least two edges. Examples of unicyclic and block graphs for which this equality fails are also presented, and two open problems are posed. Finally, we established a sharp lower bound for the mutual-visibility number of the iterated line graph of a tree and proposed an open problem. The results presented here provide a complete characterization of mutual visibility sets in trees and form a foundation for further study in more general graph classes.
\bibliographystyle{plainurl}
\bibliography{cas-refs}
\end{document}